\documentclass[12pt,dina4]{article}
\usepackage{theorem,latexsym,amsfonts,amssymb,makeidx,dsfont,epsfig,ifthen,fancyhdr,enumerate}

\newcommand{\N}{{\mathbb N}}
\newcommand{\Z}{{\mathbb Z}}

\newcommand{\R}{{\mathbb R}}
\newcommand{\C}{{\mathbb C}}

\renewcommand{\P}{{\mathbb P}}

\newcommand{\m}{\mathfrak{m}}

\newcommand{\V}{\mathcal{V}}

\newcommand{\lra}{\longrightarrow}

\newcommand{\ra}{\rightarrow}

\theoremstyle{plain}
\theorembodyfont{\upshape}
\newtheorem{definition}{Definition}[section]
\newtheorem{remark}[definition]{Remark}

\newtheorem{example}[definition]{Example}
\theorembodyfont{}

\newtheorem{theorem}[definition]{Theorem}

\newtheorem{corollary}[definition]{Corollary}
\newtheorem{conjecture}[definition]{Conjecture}
\newtheorem{proposition}[definition]{Proposition}
\newenvironment{proof}[1][]{\par\noindent{\emph{Proof#1: }}}{\hfill $\Box$\par\medskip}

\begin{document}

\title{Gorenstein-duality for one-dimensional almost complete intersections -- with an application to non-isolated real singularities}
\author{Duco van Straten and Thorsten Warmt}
\maketitle
\begin{abstract}
We give a generalization of the duality of a zero-dimensional
complete intersection to the case of one-dimensional almost complete intersections, which results in a {\em Gorenstein module} $M=I/J$. In the real case the 
resulting pairing has a signature, which we show to be constant under flat 
deformations. In the special case of a non-isolated real hypersurface
singularity $f$ with a one-dimensional critical locus, we relate 
the signature on the jacobian module $I/J_f$ to the Euler characteristic of 
the positive and negative Milnor fibre, generalising the result for isolated 
critical points. An application to real curves in $\P^2(\R)$ of even degree is given.
\end{abstract}


\section{Introduction}
The algebraic determination of the number of real roots of a polynomial has
a long history going back at least to Descartes. Of particular relevance 
are the methods of Sylvester and Hermite that determine the number of real 
roots as the {\em signature of an associated  quadratic form}. For a nice account 
of the classical approaches we refer to \cite{weber}.
 
In a similar spirit, the celebrated theorem of  Eisenbud-Levine \cite{eisenbudlevine}, and Khimshiashvilli \cite{khimshiashvilli} provides an algebraic method to determine the local degree of a finite map germ $F:(\R^n,0)\lra (\R^n,0)$
with component functions $f_1,f_2,\ldots,f_n \in P:=\R[[x_1,x_2,\ldots,x_n]]$. 
One considers the local $\R$-algebra
\[ {\cal A}_F:=P/(f_1,f_2,\ldots,f_n),\]
which has finite dimension precisely  when $f_1,f_2,\ldots,f_n$ form a 
regular sequence in $P$. Morever, in that case
${\cal A}_F$ is a {\em Gorenstein ring}: if we let
\[ h:=\left| \begin{array}{ccc} \partial_1f_1&\ldots&\partial_n f_1\\
\vdots& &\vdots\\\partial_1 f_n& \ldots&\partial_n f_n \\
\end{array} \right|\]
and choose any linear form $\phi: {\cal M} \lra \R$ with
$\phi(h) \neq 0$, then the pairing
\[B_{\phi}: {\cal A}_F \times {\cal A}_F \lra \R,\;(a,b) \mapsto \phi(a\cdot b)\]
is non-degenerate.

\begin{theorem} (Eisenbud-Levine \cite{eisenbudlevine}, Khimshiashvilli \cite{khimshiashvilli}) 
If $\phi(h) >0$ then
\[ Signature(B_{\phi})=Degree(F,0)\]
\end{theorem}

The theorem is a result of key importance and has been the starting point of many subsequent works. We mention a few of the applications and generalisations.

Consider an isolated complete intersection curve $C=f^{-1}(0)$, where
$f:(\R^n,0) \lra (\R^{n-1},0)$. According to Aoki, Fukuda and Nishimura 
\cite{aokifukudanishimura}, one can compute the {\em number of real branches of
$C$} as follows: consider $g:=x_1^2+x_2^2+\ldots+x_n^2$ and let
\[Jac(g):=\left| \begin{array}{ccc} \partial_1 g&\ldots&\partial_n g\\
\partial_1 f_1&\ldots&\partial_nf_1\\\vdots&&\vdots\\\partial_1f_{n-1}&\ldots&\partial_nf_{n-1}\\ \end{array} \right|\]
On the $\R$-algebra ${\cal A}_{f,Jac(g)}:=P/(f_1,\ldots,f_{n-1},Jac(g))$
one defines as above a pairing $B_{\phi}$.

\begin{theorem} (Aoki-Fukuda-Nishimura)
\[ Signature(B_{\phi})=\textup{Number of real branches of}\;\; C\]
\end{theorem}

This result was further generalised to the case of
arbitrary Gorenstein curve singularities $C$ in
\cite{jamesduco}.

In a similar vein, the work \cite{sza1} associates to a polynomial mapping 
$f:\R^n \lra \R^k$ an $\R$-algebra ${\cal A}$ with a quadratic form $B_{\phi}$
such that
\[ Signature(B_{\phi})=\chi(f^{-1}(0)),\]
and further variations can be found in \cite{sza2}, \cite{sza3}, \cite{dutertre}.\\

Another type of application is to the topology of real hypersurface singularities. A function $f \in P$ defines an isolated hypersurface singularity precisely
when the partial derivatives $f_i:=\partial f/\partial x_i$ form a regular
sequence in $P$. In this case the algebra ${\cal A}$ is nothing but the (real) 
Milnor algebra $P/J_f$, where $J_f$ is the jacobian ideal of $f$. The degree 
of $F:=(\partial_1f,\ldots,\partial_nf)$ is the Poincar\'e-Hopf index of the 
gradient vector field of $f$.

The Milnor ring is the most important algebraic invariant of the singularity
$f$. Its dimension $\mu(f):=\dim(P/J_f)$ is the Milnor number, which equals the
dimension of the cohomology $H^{n-1}(F)$ of the Milnor fibre $F$, \cite{agv}, 
\cite{milnor}.
In the real case one can also consider the real Milnor fibres:
\[ F_{\epsilon,\eta}:=\{ x \in \R^n \;|\; \|x\| \le \epsilon, f(x)=\eta\}, \;\; 0 < \epsilon \ll 1,\;\; 0 < \eta \ll \epsilon\]
For $\eta >0$, we put $F_{+}:=F_{\epsilon,\eta}$ and $F_{-}:=F_{\epsilon,-\eta}$
and call them the {\em positive} and {\em negative} Milnor fibers of $f$.

\begin{theorem} (\cite{arnold}) Let $f \in \R\{x_1,x_2,\ldots,x_n\}$
have an isolated critical point at the origin. Then: 
\[Signature(f)=-\widetilde{\chi}(F_{-})=(-1)^{n-1}\widetilde{\chi}(F_{+}),\] 
where $\widetilde{\chi}$ denotes the reduced Euler characteristic.
\end{theorem}

Although the positive and negative Milnor fibre have, in general,
a quite different topology, it is a simple but remarkable fact that their
reduced Euler characteristics are the same, up to a sign.
The real $D_4$ surface singularity defined by $x(x^2-y^2)+z^2=0$ may
exemplify this.
\begin{center}

\begin{tabular}{ccc}
\epsfig{figure=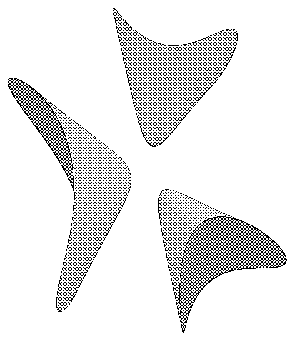}&\epsfig{figure=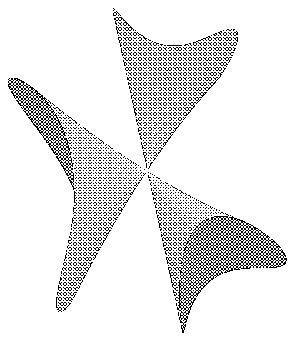}&\epsfig{figure=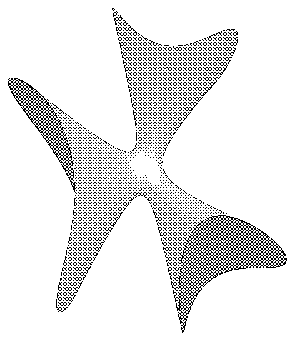}\\
negative Milnor fibre&$D_4$ singularity&positive Milnor fibre\\
\end{tabular}

\end{center}

In this paper we give a partial generalisation of this result to
the case where the sequence $f_1,f_2,\ldots,f_n$ defines a {\em one-dimensional
locus}. In this case one says the the ideal $J=(f_1,\ldots,f_n)$ defines
an {\em almost complete intersection}. We show that the torsion sub-module
$M=I/J$ of $P/J$ is a {\em Gorenstein module}. 
An isomorphism  $\phi: P \stackrel{\cong}{\lra} \Omega^n_P$ 
determines a natural pairing 
\[B_{\phi}: I/J \times I/J \lra \R\] 
Its signature is an invariant of the real topology of the situation. 

In the case of partial derivatives  $f_i=\partial f/ \partial x_i$ of a
function $f$ with a one-dimensional critical locus, such modules $I/J_f$
were also considered by Ruud Pellikaan \cite{pellikaan1} under the name 
of {\em jacobian module}. The case where $I$ is a radical ideal defining
a reduced curve, corresponds to $f$ having {\em transverse type $A_1$}.
We proof the following theorem

\begin{theorem} Assume that $f \in \R\{x_1,x_2,\ldots,x_n\}$ has one-dimensional
critical locus and transverse type $A_1$. Assume furthermore that either
$f$ has a morsification or that $n=3$. Then
 
\[2 Signature(B_{\phi})=-\widetilde{\chi}(F_{-})+(-1)^{n-1}\tilde{\chi}(F_{+})\]
wher as before $F_{\pm}$ are the positive and negative Milnor fibres of 
$f$ and $\widetilde{\chi}$ denotes the reduced Euler characteristic.
\end{theorem}
Note that contrary to case of isolated singularities, the two Euler
characteristics appearing at the right hand side are in general no 
longer equal up to a sign.
It appears that the above statement has a much broader range of validity,
and holds for many other transverse types of singularities. However, as stated,
 it is not true in full generality, but we were not able to identify the 
precise limits of its validity.\\ 

A nice application arises in the situation of a homogeneous polynomial 
$f \in \R[x,y,z]$ of even degree. It defines a curve $C =\{f=0\} \subset \P^2(\R)$, whose complement consists of a part $V_{+}$ where $f>0$ and a part
$V_{-}$ where $f<0$.

\begin{theorem} If $f \in \R[x,y,z]$ defines a curve with only ordinary
double points, then
\[Signature(B_{\phi}(f))=\chi(V_+)-\chi(V_{-})\]
\end{theorem}

The structure of the paper is as follows. After reviewing the duality
in the complete intersection case, we explain the emergence of
the Gorenstein module $I/J$ for one-dimensional almost complete
intersections and explain how the pairing behaves in a relative situation.
Then we revies some material about singularities with one-dimensional
singular locus that provide a rich source of examples and explain for special
situations the meaning of the resulting signature in the real case,

\section{Duality for zero-dimensional complete intersections}

Let $P=K[[x_1,x_2,\ldots,x_n]]$ or any other regular local $K$-algebra
of dimension $n$. Let there be given a sequence of elements 
\[ f_1,\ldots,f_n \in P\]
and let $J$ be ideal generated by them.
We denote the {\em Koszul-complex}  associated to the 
sequence ${\bf f}:=f_1,f_2,\ldots,f_n$ 
by
\[K_{\bullet}({\bf f}):=0 \rightarrow K_n \rightarrow K_{n-1} \rightarrow
 \cdots \rightarrow K_2 \rightarrow K_1 \rightarrow K_0 \rightarrow 0 \]
Its terms are $K_p:= \bigwedge^p P^n$, so that $K_0=P$.
 The differentials are induced by 
sending the $i$-th basis vector $e_i$ of $P^n$ to $f_i \in P$.
We denote its homology groups by
\[H_k({\bf f}):=H_k(K_{\bullet}({\bf f}))\]  
The following result is well-known (see e.g. \cite{brunsherzog}, thm. 1.6.17 and thm. 2.1.2).

\begin{proposition}\label{vanishing} Let $V(J) \subset Spec(P)$ the variety defined by the ideal $J$ and let $d:=dim(V(J))$. Then  $H_k({\bf f})=0$ for $k >d$ and  $H_d({\bf f}) \neq 0$ 
\end{proposition}

Let us first look at the case where $\dim(J)=0$. Then the above
result tells us that $H_0({\bf f})=P/J$ and $H_k({\bf f})=0$ for $k>0$. 
So ${\bf f}$ is a regular sequence and  the Koszul complex provides a free 
resolution of $P/J$ as a $P$-module.
As the transpose of the map $ K_n \lra K_{n-1}$ can be identified, up to
signs, with the map $K_1 \lra K_0$, we obtain an isomorphism
\[ Ext^n_P(P/J,P) \cong P/J\]

Recall Grothendiecks {\em local duality theorem} (see e.g. \cite{brunsherzog}, thm. 3.5.8):

\begin{theorem}\label{localduality} 
Let $P$ be a $n$-dimensional local ring with maximal ideal ${\m}$
and dualizing module $\omega_P$. For any finitely generated $P$-module $M$,  
there exists a natural non-degenerate pairing
\[H^i_{\m}(M) \times Ext^{n-i}_P(M,\omega_P) \lra H^n_{\m}(\omega_P) \stackrel{tr}{\lra} K \]
\end{theorem}
In particular, for $M=P/J$ and $i=0$ we obtain a non-degenerate
pairing
\[ P/J \times Ext^n_P(P/J,\omega_P) \lra K\]
as $H^0_{\m}(P/J)=P/J$.
The choice of an isomorphism $\phi: P \stackrel{\cong}{\lra} \omega_P$ induces isomorphisms
\[Ext^n_P(P/J,\omega_P)\cong Ext^n_P(P/J,P)\cong P/J\]
then hence provides us with a perfect pairing
\[B_{\phi}: P/J \times P/J \lra K \]
As $B_{\phi}$ is $P$-linear, it follows that it factors over the 
multiplication map and is of the form $B_{\phi}(a,b)=\beta(a\cdot b)$
for some linear form
\[\beta: P/J \lra K\]
The linear space $ker(\beta)$ is the {\em socle} of $P/J$, that is, its 
unique minimal ideal. 
There is a classical result of Scheja and Storch that states that this socle has  a { canonical generator}.

\begin{theorem} \cite{schejastorch}
If $char(K)=0$ and $x_1,\ldots,x_n$ are generators of the maximal ideal, then the Jacobian determinant

\[ h:=\left|\begin{array}{ccc} \frac{\partial f_1}{\partial x_1}& \ldots &\frac{\partial f_1}{\partial x_n}\\
\vdots&&\vdots\\
\frac{\partial f_n}{\partial x_1} &\ldots& \frac{\partial f_n}{\partial x_n}\\
 \end{array}\right|\] 
is a generator for the socle of $P/J$.
\end{theorem}

\begin{remark}
In the case where $f \in P$ and $f_i=\partial f/\partial x_i$, the complex 
can be identified (up to some signs) with the $df \wedge$-complex
\[ 0 \lra P \stackrel{df \wedge }{\lra}\Omega^1 \stackrel{df\wedge}{\lra} \ldots \lra \Omega^{n-1}
\stackrel{df \wedge}{\lra}\Omega^n \lra 0\]
and the natural pairing takes the form
\[ B: \Omega^n/df\wedge \Omega^{n-1} \times\;\; \Omega^n/df\wedge \Omega^{n-1} \lra K\]
It is usually called {\em residue pairing} and has an analytic expression
as
\[ B(P\omega,Q\omega)=\frac{1}{(2\pi i)^n}\int_{T_{\epsilon}} \frac{PQ\omega}{\partial_1f\ldots \partial_n f}\]
where $\omega:=dx_1\wedge dx_2 \wedge \ldots \wedge dx_n$ and $T_{\epsilon}:=\{x\;|\;|\partial_i f|=\epsilon\}$, $0<\epsilon \ll 1$. 
In the papers \cite{giventalvarchenko}, \cite{varchenko} the relation between
this pairing and the Poincar\'e-pairing in the cohomology of the Milnor fibre
is described. The pairing is also the first in a sequence of
{\em higher residue pairings}, introduced by K. Saito, which express the 
self-duality of the Gau{\ss}-Manin system of the singularity, \cite{saito}.
\end{remark}

\section{Homology and Cohomology}

We have seen that in the case of a regular sequence ${\bf f}$,
the origin of the pairing on $P/J$ lies in the self-dual nature of the 
Koszul complex. We now investigate in general what consequence this 
self-duality has for the Koszul homology groups. 

\begin{proposition} \label{selfdual} (c.f. \cite{eisenbud}, prop. 17.15)
Let ${\bf f}=f_1,f_2,\ldots,f_n$ be a sequence of elements in a ring $P$.
There are isomorphisms
\[ \alpha_p: H_p\left({\bf f}\right)\cong H^{n-p}\left({\bf f}\right), \;\;p=0,1,\ldots,n,
\] 
where $H^p({\bf f}):= H^p(Hom_P\left(K_{\bullet}\left({\bf f}\right),P\right))$
is the Koszul cohomology.
\end{proposition}
\begin{proof}
Using the basis $e_1,e_2,\ldots,e_n$ for the free module $P^n$ and the
dual basis $\phi_1,\ldots,\phi_n$ for the dual module $Hom_P(P^n,P)$, one
obtains  basis elements $e_I:=e_{i_1}\wedge e_{i_2} \wedge \ldots \wedge e_{i_p}$ for
$\bigwedge^p P^n$ and $\phi_I=\phi_{i_1}\wedge\ldots \wedge\phi_{i_p}$, $(I=(i_1,i_2,\ldots,i_p))$ for the dual module $Hom_P(\bigwedge^p P^n,P)$. 
Define isomorphisms 
$\alpha_p:\bigwedge P^p \lra (\bigwedge^pP^n)^*:=Hom_P(\bigwedge^{n-p}P^n,P)$
by setting $\alpha_p(e_I)=sign(\sigma)\phi_J$
where $J$ is the sequence of indices complementary to $I$ and $sign(\sigma)$
is the sign of the permutation that puts the sequence $(I,J)$ into $(1,2,\ldots,n)$. One verifies that in this way one obtains a mapping between complexes
$\bigwedge^{\bullet} P^n$ and $(\bigwedge^{\bullet} P^n)^*$. 

%
We refer to \cite{eisenbud}, section 17.4 for more details.
\end{proof}

\begin{remark} The modules $H_i(\bf f)$ and $H^i(\bf f)$ depend only on the
ideal $J$ generated by $f_1,f_2,\ldots,f_n$. Note however, that changing 
the order changes $\alpha_p: H_p({\bf f})\lra H^{n-p}({\bf f})$ by a sign. 
\end{remark}

In general cohomology is also {\em dual} to homology, in the 
following sense.
\begin{theorem}
Let $P$ be a ring, $(F_{\bullet},\partial_{\bullet})$ a complex of free 
$P$-modules and $M$ a $P$-module. 
Then there exists a spectral sequence with
$$E_2^{p,q}=Ext^p_P(H_q(F_{\bullet}),M)\Longrightarrow H^{p+q}(Hom_P(F_{\bullet},M)).$$
\end{theorem}
\begin{proof} This is of course classical, see  for example \cite[section XVI.2]{eilen}. It can be shown as follows: consider an injective resolution $(G^{\bullet},\delta^{\bullet})$ of
$M$ and the double complex $C_{\bullet,\bullet}$ with terms $C_{p,q}=Hom_P(F_{q}, G^{p})$ and differentials induced from $\partial$ (going up. increasing $q$) and $\delta$ (going right, increasing $p$).
%
The spectral sequence obtained by first taking $\delta$ and then $\partial$
degenerates at $E_2$ and has $H^p(Hom_P(F_{\bullet},M))$ at spot $(p,0)$ as
homology. The other spectral sequence obtained by first taking $\partial$ and followed
by $\delta$ has as $E_2$-term
\[
\begin{array}{c|cccc}
q&Hom_P(H_q, M)&Ext^1_P(H_q, M)&\ldots&Ext^p_P(H_q, M)\\
\vdots&\vdots&\vdots& &\vdots\\
1&Hom_P(H_1, M)&Ext_P^1(H_1, M)&\ldots&Ext^p_P(H_1, M)\\
0&Hom_P(H_0, M)&Ext_P^1(H_0, M)&\ldots&Ext^p_P(H_0, M)\\
\hline
&0&1&\ldots&p
\end{array}
\]
where $H_i:=H_i(F_{\bullet})$.
\end{proof}

By taking $M=P$, this spectral sequence can be used to connect the homology and 
the cohomology of a complex. As a corollary, we note the two special situations 
which lead to exact sequences.

\begin{corollary}\label{corollaryspek}
(i) If the groups $Ext^k_P(H_i(F_{\bullet}),P)$ vanish for $k>1$, the spectral sequence collapses to exact sequences
$$
0\lra Ext^1_P(H_k(F_{\bullet}),P)\lra H^{k+1}(F^{\bullet})\lra Hom_P(H_{k+1}(F_{\bullet}),P)\lra 0.
$$
(``universal coefficient theorem'')\\

(ii) If the homology groups $H_k:=H_k(F_{\bullet})$ vanish for $k>1$, we get a long exact sequence 
\[\cdots \ra Ext^k_P(H_0,P)\ra H^k \ra Ext^{k-1}_P(H_1,P)\stackrel{d_2}{\ra} Ext^{k+1}_P(H_0,P)\ra \cdots\]
where $H^i:=H^i(F^{\bullet})$.
\end{corollary}

We specialise the above to the case of an almost complete intersection.

\begin{definition}
A sequence ${\bf f}=f_1,f_2,\ldots,f_n$ of elements in a ring $P$ is said to
defines an {\em almost complete intersection} if $codim(V(J))=n-1$. 
\end{definition}
 
It follows from \ref{vanishing} that in this case one has two non-vanishing Koszul homology groups, $H_0({\bf f})=P/J$ and one further module
$H_1({\bf f})$.

\begin{proposition} \label{aci}
Let ${\bf f}=f_1,f_2,\ldots,f_n$ define an almost complete intersection in $P$ 
and put $H_i:=H_i({\bf f})$. Then:\\
(i)
\[ Ext^{n-1}_P(H_0,P)\cong H_1\]
(ii) There is an exact sequence
\[
0 \lra Ext^{n}_P(H_0,P) \lra H_0 \lra Ext^{n-1}_P(H_1,P) \stackrel{d_2}{\lra} Ext_P^{n+1}(H_0,P) \lra 0\]
(iii) For $k \ge n$ there are isomorphisms
\[ Ext^{k}_P(H_1,P) \stackrel{d_2}{\lra} Ext_P^{k+2}(H_0,P)\]
\end{proposition}
\begin{proof}
The self-duality \ref{selfdual} of the Koszul complex gives
$H^{n-1}\left({\bf f}\right)=H_1\left({\bf f}\right)$ and 
$H^n\left({\bf f}\right)=H_0\left({\bf f}\right)$ as only non-vanishing
Koszul cohomology groups. As both modules are supported on $V(J)$, 
it follows from Ischebeck's lemma  that 
$Ext^p_P(H_i({\bf f}),P)=0$, for $p<n-1$. The theorem follows then
immediately from the above long exact sequence \ref{corollaryspek} (ii).
\end{proof}

\section{One-dimensional almost complete intersections}
We now make the further assumption, that $dim(P)=n$, so that
the almost complete intersection  ${\bf f}:=f_1,\ldots,f_n$ gives
an ideal $J=(f_1,f_2,\ldots,f_n)$ that defines a one-dimensional
locus $V(J) \subset Spec(P)$.

\begin{corollary}\label{koszhom}
Let ${\bf f}:=f_1,\ldots,f_n$ define a one-dimensional almost complete 
intersection. Then:\\
(i)
\[ Ext^{n-1}_P(H_0({\bf f}),P)\cong H_1({\bf f}),\;\;Ext_P^{n}(H_1({\bf f}),P)=0.\]
(ii) There is an exact sequence
\[
0 \lra Ext^{n}_P(H_0,P) \lra H_0({\bf f})\lra Ext^{n-1}_P(H_1({\bf f}),P) \lra 0\]
\end{corollary}

\begin{proof}
As $P$ is a assumed to be a regular local ring of dimension $n$, we have
$Ext^k_P(-,P)=0$ for $k > n$. So the result follows from the above proposition
\ref{aci}
\end{proof}

\begin{definition}\label{imodj} For a one-dimensional almost complete intersection defined by ${\bf f}$ we put 
\[ M:=Ext^{n}_P(H_0({\bf f}),P) \]
As $Ext^{n}_P(H_0({\bf f}),P) \subset H_0({\bf f})=P/J$ we see that 
\[M=I/J\]
for an ideal $I \subset P$.
\end{definition}
In other words, $I$ is defined by having an isomorphism of exact
sequences
\[
\begin{array}{cccccccccc}
0 &\lra &Ext^{n}_P(H_0({\bf f}),P)& \lra &H_0({\bf f})&\lra& Ext^{n-1}_P(H_1({\bf f}),P)& \lra& 0\\
&&\cong\downarrow&&\cong \downarrow&&\cong\downarrow&&\\
0&\lra&I/J&\lra&P/J&\lra&P/I&\lra& 0\\
\end{array}
\]

\begin{proposition} \label{saturation} 
The ideal $I$ is equal to the {\em saturation} of $J$ with respect to the maximal ideal:
\[
I=\bigcup_k(J:\mathfrak{m}^k) \supset J
\] 
Thus the module $M=I/J$ is the \emph{$\m$-torsion submodule of $P/J$}: 
$$M=I/J=H^0_{\m}(P/J)$$
\end{proposition}
\begin{proof}
Because $Ext^{n}_P(H_1({\bf f}),P)=0$, we see that the module $H_1({\bf f})$
has $Ext^{n-1}_P(H_1({\bf f}),P)=P/I$ as the only non-vanishing $Ext$. Hence
it is a Cohen-Macaulay $P$-module of dimension one. It follows that 
$H^0_{\m}(P/I)=0$
and thus the artinian module $M=Ext^{n}_P(H_0({\bf f}),P)$ is equal to
$H^0_{\m}(P/J)$. Hence $I$ is nothing but the saturation of $J$ with
respect to ${\m}$. 
\end{proof}

\begin{remark}
In fact this provides an easy was to compute $I$ and $I/J$ using a computer
algebra system like {\sc Singular} or {\sc Macaulay}. The ring $P/I$ obtained
by dividing out the {\bf m}-primary torsion of $P/J$ is sometimes called the
{\em Cohen-Macaulayfication} of $P/J$.
\end{remark}

Using $I/J$ and $P/I$ we can give a slight reinterpretation of the
$Ext$ of Koszul homology.

\begin{theorem}\label{extiso}
The map $P/J \lra P/I$ induces an isomorphism
\[Ext^{n-1}_P(P/I,P) \cong Ext^{n-1}_P(P/J,P)\]
and the map $I/J \lra P/J$ induces an isomorphism
\[Ext^{n}_P(P/J,P) \cong Ext^n_P(I/J,P)\]
\end{theorem}
\begin{proof}
Apply $Hom_P(\bullet,P)$ to the short exact sequence of $P$-modules
$0\lra I/J\lra P/J\lra P/I\lra 0$.
As $I/J$ is $\m$-torsion and $P/I$ is ${\m}$-torsion free the result
follows. 
\end{proof}

We now have all the ingredients for the following central result.

\begin{proposition}\label{isoextij}
Let the sequence ${\bf f}=f_1,\ldots,f_n$ define a one-dimensional
almost complete intersection $J$, and $M=I/J=H^0_{\bf m}(P/J)$. Then
there is an isomorphism
\[Ext^n_P(M,P) \cong M.\]
\end{proposition}

\begin{proof}
Combining theorem \ref{extiso} and theorem \ref{koszhom}(i) we obtain an
isomorphism
\[Ext^{n-1}_P(P/I,P)\cong Ext^{n-1}_P(P/J,P) \cong H_1({\bf f}),\]
showing that $H_1({\bf f})$ is isomorphic to the dualising module of $P/I$. Because $P/I$ is Cohen-Macaulay, we have
\[Ext^{n-1}_P(H_1({\bf f}),P)\cong Ext^{n-1}_P(Ext^{n-1}_P(P/I,P),P) \cong P/I.\]

Combining this with \ref{koszhom} (ii) we see that
\[Ext^n_P(P/J,P) \cong ker(P/J \ra P/I)=I/J.\]

But by \ref{extiso} we also have $Ext^n_P(I/J,P) \cong Ext^n_P(P/J,P)$ and
as a result we obtain an isomorphism
\[ Ext^n_P(I/J,P) \cong I/J \]
\end{proof}

As a corollary one finds:

\begin{theorem}\label{pairing}
The module $M=I/J=H_{\m}^0(P/J)$ is an artinian Gorenstein module.
The choice of an isomorphism $\phi: P \stackrel{\cong}{\lra} \omega_P $ 
determines a  non-degenerate pairing
\[
B_{\phi}: M \times M \lra K.
\]
\end{theorem}
\begin{proof}
From local duality \ref{localduality} we obtain a non-degenerate pairing
$$
M \times Ext^n(M,\omega_P) \lra H^n_{\m}(\omega_P)\stackrel{tr}{\lra} K
$$
The choice of an isomorphism $\phi: P \stackrel{\cong}{\lra} \omega_P $ 
determines an isomorphism $Ext^n_P(I/J,P) \cong Ext^n_P(M,\omega_P)$ an
thus  a non-degenerate pairing
$B_{\phi}: M \times M \lra K.$
\end{proof}

\begin{remark}
(i) The material of this section should be rather well-known. For example,
the isomorphism $H_1=Ext_P^{n-1}(P/I,P)$ can be found in \cite{pellikaan3},
but the self-duality of $I/J$ seems to have escaped attention. It is reflected
in the readily observed symmetry of the Hilbert-Poincar\'e polynomial.\\
(ii) Using a computer algebra system, the duality pairing can in be calculated by running through the appropriate sequences and isomorphisms. This was implemented in {\sc Singular} and is described in some detail in the thesis of the second author, \cite{egodiss}. As a simple example, for ${\bf f}=(xy,x^2) \subset K[[x,y]]$,  $\phi(1)=dx\wedge dy$ one has $I=(x)$, $\dim(I/J)=1$ and $B_{\phi}([x],[x])=1$.\\ 
(iii) The above can be generalised, with almost identical proof,
to the case of a one-dimensional complete intersection in an arbitrary
Gorenstein ring $P$. If $\omega_P$ denotes the dualising module of $P$, 
then one obtains a natural pairing
\[B: {\cal M} \times {\cal M} \lra K\]
where ${\cal M}=H_{\bf m}^0(\omega_P/J\omega_P)$.\\
(iv) In the special case of hypersurfaces singularity with one-dimensional 
singular locus,  these modules $I/J$ were studied in \cite{pellikaan1} under the name of {\em Jacobian modules} and
play a r\^ole that can be compared to that of the Milnor ring in the isolated case. See also section 6.\\
(v) In the case of hypersurface singularities with one-dimensional singular locus it is again more natural to consider the cohomology $H^n$ and $H^{n-1}$ of
the $df \wedge$-complex. On the $\m$-torsion submodule of $H^n=\Omega^n/df \wedge \Omega^{n-1}$ (isomorphic to $I/J$) there again is a pairing that does not involve any choice.
In this situation there is also a map $d:H^{n-1}\lra H^n$ induced by exterior
differentiation. That map plays a role in the Gauss-Manin system of $f$, \cite{vanstraten}.\\
(vi) Hypersurfaces in projective space with isolated singularities 
correspond to homogeneous singularities with one-dimensional singular locus. 
The module $I/J$ plays a role in the Dwork-Griffiths description of 
the Hodge-pieces of the cohomology. For example, let 
$f \in \C[x_0,x_1,x_2,x_3,x_4]$ be the equation of a projective 
threefold $X=V(f) \subset \P^4$ of degree $d$ with only nodes as 
singularities, and let $\pi:Y \lra X$ a small resolution.
Then the degree $a=2d-5$ part $(I/J)_a$ of $I/J$ can be identified with 
$H^{2,1}:=H^1(\Omega^2_Y)$ via
\[A \lra \pi^*(Res(\frac{A\Omega}{f}))\]
where $\Omega:=\iota_E(dx_0 \wedge \ldots \wedge dx_4)$.
Similarly, the degree $b=3d-5$ part of $I/J$ is identified with $H^{1,2}:=H^{2}(\Omega^1_Y)$, making a commutative diagram
\[
\begin{array}{ccccc}
(I/J)_a & \times & (I/J)_b & \lra &\C\\
\downarrow &&\downarrow&&\cong \downarrow\\
H^{2,1}&\times&H^{1,2}&\stackrel{\cup}{\lra}&H^{3,3}\\
\end{array}
\] 
For details we refer to \cite{dimcasaitowotzlaw}.
\end{remark}

\section{Behaviour under flat deformations}

We now study the behaviour of the module $I/J$ under deformation. By this we
mean that we let the almost complete interesection $f_1,\ldots,f_n$ depend on 
additional parameters. If we require $H_0=P/(f_1,f_2,\ldots,f_n)$ to
deform in a flat way, then the same will be true for $I/J$.\\

\begin{definition} 
Consider a local ring $S$ with maximal ideal $\m_S$ and $K=S/\m_S$
and let $R$ be a flat $S$-algebra such that $P=K \otimes_S R$. 
A sequence ${\bf F}:=F_1,F_2,\ldots,F_n \in R$ is called relative
almost complete intersection if\\
1) \[f_i:=F_i\;\; \textup{mod}\;\;\m_S\] 
2) $H_0({\bf F})$ is $S$-flat.
\end{definition}

\begin{proposition} \label{flat}
If ${\bf F}$ defines a relative almost complete intersection, then 
$Ext_R^n(H_0({\bf F}),R)$ is free as $S$-module and $Ext_R^k(H_0({\bf F}),R)=0$ for $k \ge n+1$.
\end{proposition}
\begin{proof} This follows from a cohomology-and-base change argument. We
assume for simplicity $S=K[[t]]$, so that we have an exact 
sequence $0\lra R \stackrel{t\cdot}{\lra} R \lra P \lra 0$.   
As ${\bf F}$ defines an almost complete intersection in $R$, we have as before
only two Koszul groups $H_0:=H_0({\bf F})$ and $H_1:=H_1({\bf F})$, and by 
assumption we have an exact sequence
\[ 0 \lra H_1 \stackrel{t\cdot}{\lra} H_1 \lra \overline{H}_1 \stackrel{0}{\lra} H_0 \stackrel{t\cdot}{\lra} H_0 \lra \overline{H}_0 \lra 0\] 
where $\overline{H}_k:=H_k({\bf f})$. Furthermore, we have the statements 
from \ref{aci}. We will show that $Ext^{n+1}_R(H_0,R)=0$, as are all further 
higher $Ext^{k}_R(H_1,R) \cong Ext^{k+2}_R(H_0,R)$, $k \ge n$. For this, note that one obtains a long exact sequence
\[ \ldots \lra E^{n+1} \stackrel{t\cdot}{\lra} E^{n+1} \lra \overline{E}^{n+1}\lra E^{n+2} \stackrel{t\cdot}{\lra} E^{n+2} \lra \overline{E}^{n+2}\lra \ldots\] 
where we put temporarily $E^k:=Ext^k_R(H_0,R)$, $\overline{E}^k:=Ext_P(\overline{H}_0,P)$. As $\overline{E}^k=0$ for $k \ge n+1$ and the modules are $S$-finite,
one concludes with Nakayama that all $E^{k}=0,k \ge n+1$. As $H_1=E^{n-1}$ it follows that we have an exact sequence
\[ 0  \lra E^{n-1} \stackrel{t\cdot}{\lra} E^{n-1} \lra \overline{E}^{n-1} \stackrel{0}{\lra} E^{n} \stackrel{t \cdot}{\lra} E^{n} \lra \overline{E}^{n}\lra 0\]
hence we find that $E^n=Ext_P^n(H_0({\bf F}),P)$ and $S$-flat, and hence $S$-free. 
\end{proof}

As in \ref{koszhom} we have
\[Ext^{n-1}_R(H_0({\bf F},R))=H_1({\bf F}),\;\;\;Ext^n_R(H_1({\bf F}),R)=0\]
\[0 \lra Ext_R^n(H_0({\bf F}),R) \lra H_0({\bf F})\lra Ext^{n-1}_R(H_1({\bf F}),R) \lra 0\]

\begin{definition} \label{relativemodule} 
In the above situation we put
\[M_S:=Ext^n_R(H_0({\bf F}),R)\]
\end{definition}

As $H_0({\bf F})=R/J_S$, we see as before from the exact sequence that
there exists an ideal $I_S \subset R$ such that $M_S=I_S/J_S$. Completely 
analoguous to \ref{isoextij} we have

\begin{proposition}\label{isoextrelativ} 
\[ Ext^n_R(M_S,R)=M_S\]
\end{proposition}
\begin{proof}
We omit the proof, that is identical to that of \ref{isoextij}.
\end{proof}
%

However, we want to understand this duality in terms of a {\em family of 
pairings}, parametrised by $S$. We can not just apply the local duality 
theorem, but rather we would  have to use the duality statement for the 
morphism $Spec(R) \lra Spec(S)$.\\
 
To treat this in an elementary way, we express higher $Ext$-groups $Ext_R^i(M,R)$ as groups $Hom_{\overline{R}}(M,\overline{R})$ of homomorphisms, where $\overline{R}$ arises from $R$ by dividing out suitable elements. In this way
we can reduce to the case of a finite ring extension and use the duality 
for a finite map, a change-of-rings isomorphism.  First we recall

\begin{theorem}\label{exthom}
Let $R$ be a ring and let  $M$, $N$ be two non-trivial finite $R$-modules.
If $\textup{ann}(M)+\textup{ann}(N)=R$, $Ext_R^i(M,N)$ is zero for all $i\in\N$. Otherwise 
$$d=\textup{depth}(\textup{ann}(M),N)$$ 
is the smalles number $i$ with  $Ext^i_R(M,N)$ not equal to zero.
If $t_1,\ldots,t_d$ is a maximal regular $N$-sequence in $\textup{ann}(M)$
and we define
$$\overline{R}=R/(t_1,\ldots,t_d), \overline{N}=N/(t_1,\ldots,t_d)N,$$
then there is an isomorphism
$$
Ext^d_R(M,N) \cong Hom_{\overline{R}}(M,\overline{N}).
$$
\end{theorem}
\begin{proof} This is well-known. The first part of the statement is essentially
\cite{brunsherzog}, 1.2.10 and the rest follows by induction from the long
exact Ext-sequence, obtained by dividing out an element.
\end{proof}

\begin{theorem} \label{pairingfamily} An isomorphism $\phi: R \stackrel{\cong}{\lra} \omega_{R/S}$ 
defines a bilinear pairinng
\[B_{\phi}: M_S \times M_S \lra S\]
This pairing is non-degenerate in the sense that the adjoint map
\[M_S \cong Hom_S(M_S,S)\]
is an isomorphism of $S$-modules.
\end{theorem}
\begin{proof} 
From \ref{isoextrelativ} there is is an isomorphism
$M_S \cong Ext^n_R(M_S,R)$. We take a maximal regular sequence 
$t_1,\ldots,t_d$ in the annihilator of $M_S$. We divide out these 
elements and obtain a factor ring $\overline{R}$ of $R$ and applying
\ref{exthom} we get an isomorphism 
$Ext^n_R(M_S,R)= Hom_{\overline{R}}(M_S,\overline{R})$. As $M_S$ is a finite 
$S$-module, the ring $\overline{R}$ is finite over $S$. From the isomorphism 
$\phi: R \stackrel{\cong}{\lra} \omega_{R/S}$ we obtain by dividing out 
$t_1,\ldots,t_d$ an isomorphism $\overline{R} \cong \omega_{\overline{R}/S}$.
Duality for the finite map $S \lra \overline{S}$ tells us
$\omega_{\overline{R}/S}=Hom_S(\overline{R},S)$. From the change-of-rings 
isomorphism we get
$$
Hom_{\overline{R}}\left(M_S,Hom_S\left(\overline{R},S\right)\right) \cong Hom_{S}(M_S,S).
$$
Combining these isomorphisms we obtain
\[Ext_R^n(M_S,R)\cong Hom_{\overline{R}}(M_S,\overline{R})\cong Hom_S(M_S,S)\]
Hence, in total we obtain a natural isomorphism 
$M_S\cong Hom_S\left(M_S,S\right)$, which can be seen as a family of 
non-degenerate pairings $B_{\phi}: M_S\times M_S\lra S.$
\end{proof}

In the case $S=K[[t]]$ we obtain a commutative diagram of the following form
\[
\begin{array}{ccccccccc}
0&\lra&M_S          &\stackrel{t\cdot}{\lra}&M_S            &\lra& M_0          &\lra&0\\
 &    &\uparrow\cong&                        &\uparrow \cong&    &\uparrow \cong&    &\\
0&\lra&Ext^n_R(M_S,R)&\stackrel{t\cdot}{\lra}&Ext^n_R(M_S,R)      &\lra& Ext^n_P(M_0,P)&\lra&0\\
 &    &\uparrow\cong&                        &\uparrow \cong&    &\uparrow \cong&    &\\
0&\lra&Hom_S(M_S,S)  &\stackrel{t\cdot}{\lra}&Hom_S(M_S,S)   &\lra& Hom_K(M_0,K) &\lra&0\\
\end{array}
\]

\section{Application to  non-isolated hypersurface singularity 
with one-dimensional critical locus}

\subsection{Hypersurfaces with one-dimensional singular locus}

Following \cite{pellikaan1}, the {\em primitive} of an ideal 
$I \subset P:=\C\{x_1,\ldots,x_n\}$ is the ideal
\[\int I:=\{f \in P\;|\;(f,\partial_1f,\ldots,\partial_nf)\subset I\}\]
This ideal arises when studying functions $f$ which contain a
specific sub-space $\Sigma:=V(I)$ inside their critical locus. 
One can pursue the classification program of singularity theory in this context
and Siersma \cite{siersma} started the investigation of 
{\em line-singularities}, which correspond to the case where $I$ is a radical 
ideal defining a line. 
The {\em extended $I$-codimension} of
 a function $f \in \int I$ is defined as
\[c_{e,I}(f):=dim_{\C}(\int I/\int I \cap J_f)\]
and one can try to classify the cases of low codimension. Here $J_f=(\partial_1f,\ldots\partial_nf)$ is the jacobian ideal of $f$, \cite{pellikaan1}.
As $J_f \subset I$, one can consider the associated {\em jacobian module} 
$I/J_f$.

We will assume from now on that $I$ is a radical ideal, 
defining a curve germ $\Sigma \subset (\C^n,0)$. On has the following basic
result.

\begin{proposition} (Pellikaan, \cite{pellikaan1}, prop. 1.7)
The following statements about a function $f \in \int I$
with jacobian ideal $J_f$ are equivalent:\\
1) $\dim(I/J_f) <\infty$\\
2) $c_{e,I}(f) < \infty$\\
3) The singular locus of $f$ is $\Sigma$ and $f$ has only $A_1$-singularities
transverse to $\Sigma\setminus\{0\}$.
\end{proposition}

Note that in this situation $I=rad(J_f)$, which is the same as the saturation
of $J_f$ with respect to $\m$. So the jacobian module $I/J_f$ is precisely the
module considered in \ref{imodj} for the sequence $f_1,f_2,\ldots,f_n$,
with $f_i=\partial f/\partial x_i$.\\
 
There is a formula expressing  $\dim I/J$ in terms of $c_{e,I}$ and
some other invariant that we explain now. The two dual exact sequences
\[ I/I^2 \stackrel{d}{\lra} \Omega \otimes \mathcal{O}_{\Sigma} \lra \Omega_{\Sigma} \lra 0\]
\[ 0 \lra \Theta_{\Sigma} \lra \Theta \otimes \mathcal{O}_{\Sigma} \stackrel{d^*}{\lra} Hom(I/I^2,\mathcal{O}_{\Sigma})\]
provide important invariants of the curve $\Sigma$.\\ 

$\bullet$ The first tangent homology is $T_1(\Sigma)=ker(d)=\int I/I^2$.
In case $I$ is a reduced complete intersection, one has $\int I=I^2$. But
for $I=(xy,yz,zx) \subset \C\{x,y,z\}$ one has $\int I=(xyz,I^2)$, so $\int I/I^2$ is one-dimensional.\\

$\bullet$ The first tangent cohomology if $T^1(\Sigma)=Coker(d^*)$. It is the
space of first order infinitesimal deformations of $\Sigma$. The normal module
 $N=Hom_{\Sigma}(I/I^2,{\cal O}_{\Sigma})$ is identified with the space of embedded deformations of $\Sigma$.\\

$\bullet$ Dualising once more, we obtain a double duality map
$ I/I^2 \lra N^*$, where $N^*:=Hom_{\Sigma}(N,P/I)$. The kernel is again $\int I/I^2$, the cokernel $N^*/I:=N^*/(I/\int I)$ is a further invariant. For $I=(xy,xz,zx)$ one finds $\dim(N^*/I)=3$.\\

$\bullet$ In the special where $\Sigma \subset (\C^3,0)$ is a space curve, 
$I$ is Cohen-Macaulay of codimension $2$. Such curves are syzygetic ($T_2(\Sigma)=0$) and unobstructed ($T^2(\Sigma)=0$). Furthermore, one has in that case
\[ \int I/I^2=Ext^1({\cal O}_{\Sigma},{\omega_{\Sigma}}),\;\;\; N^*/I=Ext^2({\cal O}_{\Sigma},{\omega_{\Sigma}})\]

$\bullet$ In \cite{dJ} and \cite{dJdJ} an invariant $VD_{\infty}(f)$ called
the virtual number of $D_{\infty}$-points was defined.
In case that $f \in I^2$ Pellikaan \cite{pellikaan2} (and more generally in \cite{dJvS1}), the expression of $f$ as a quadratic form in the generators of $I$ 
can be used to define a {\em transverse Hessian map}
\[H: N \times N \lra {\cal O}_{\Sigma}\] 
which defines by transposition a map
\[h:  N \lra N^*\]
which has a finite cokernel in case $f$ has transverse $A_1$-singularities.

\begin{theorem} (\cite{pellikaan2}, \cite{dJvS1}) Let the radical ideal
define a curve $\Sigma$ and let $f \in \int I$ have transverse type $A_1$.
Assume furthermore that $\Sigma$ is smoothable, with $T_2(\Sigma)=T^2(\Sigma)=0$. Then:
\[\dim(I/J_f)=c_{e,I}(f)+\dim(T^1(\Sigma))+VD_{\infty}(f)-\dim(\int I/I^2)\]
\[VD_{\infty}(f)=\dim(N/h(N))-\dim(N^*/I)+\dim(\int I/I^2)\]
\end{theorem}

These invariants have furthermore an interpretation in terms of deformation
theory. An {\em admissible deformation} of the pair $f, \Sigma$ over a base 
$S$ consist a
flat deformation of $\Xi \lra S$ of $\Sigma$, together with a 
deformation of the $F:(\C^n \times S,0)\lra (\C \times S,0)$, such that
$\Xi$ is contained in the critical locus of $F$. We refer to
\cite{pellikaan2} and \cite{dJvS1} for more details. There is a notion
of good represenetatives for the germs involved and such a good representative 
of a one-parameter deformation $F$, $\Xi$ is called a {\em morsification}, if 
for $t \neq 0$ the curve $\Xi_t$ is smooth and
the critical points of $F_t: X \lra \C$ are of the simplest possible type, 
namely $A_1$, $A_{\infty}$ or $D_{\infty}$. 
If the curve $\Sigma$ defined by $I$ is smoothable, and $f \in I^2$, then
morsifications do exist (\cite{pellikaan2}, prop. 3.4). But even for
functions of three variables morsifications do not alway exist, as triple
points $f=xyz$ generically occur. Allowing for these leads to the notion
of {\em disentanglement}, but even these do not always exists, as projections
of non-smoothable normal surfaces singularities to $(\C^3,0)$ show. The following  result can be found in \cite{pellikaan2}.

\begin{theorem}(Pellikaan, \cite{pellikaan2}, prop. 2.19)
If $f$ posesses a morsification, then
\[ \dim(I/J_f)=\# A_1 +\# D_{\infty}\]

where $\# A_1$ and $\# D_{\infty}$ denote the number of these singularities
appearing in a morsification.
\end{theorem}
\begin{proof} (sketch) Consider a morsification over $S$.
The construction of $I/J$ in the relative case \ref{relativemodule} sheafifies 
in an obvious way to produce a sheaf ${\cal M}_S={\cal I}_S/{\cal J}_S$. Result 
\ref{flat} implies that $\pi_*({\cal M}_S)$ is a free ${\cal O}_S$-module of
finite rank, equal to $\dim(I/J)$. The freeness implies that this
number is also equal to sum of the local contributions for the function $F_t$.
A local calculation shows that both an $A_1$- and $D_{\infty}$-point give
a contribution of $1$, hence the fomula follows.
\end{proof}

In a similar vein, T. de Jong has shown that for a disentaglement of $f$ one
has 
$VD_{\infty}(f)=\# D_{\infty}-2\# T$
where $T$ denotes the triple point $f=xyz$.

\subsection{A signature theorem for functions}
We now consider a real function $f \in \R\{x_1,x_2,\ldots,x_n\}$ with jacobian
ideal $J_f$ with $\dim(V(J_f))=1$. We let $I$ be the saturation of $J_f$ and consider the jacobian module $I/J_f$.
Using the sequence ${\bf f}=f_1,f_2,\ldots,f_n$ and the isomorphism 
$\phi: P \lra \Omega^n$ given by $1 \mapsto dx_1\wedge dx_2 \wedge
\ldots dx_n$. From theorem \ref{pairing} we obtain a non-degenerate
pairing 
\[B_f: I/J_f \times I/J_f \lra \R\]

\begin{definition} The signature invariant of $f$ is 
\[ \sigma(f):=Signature(B_f) \in \Z\]
\end{definition}
This is indeed an invariant of $f$; it equals the signature of the
canonical pairing $B:M_f \times M_f \lra \R$, where 
$M_f:=H_{\bf m}^0(\Omega^n/df\wedge \Omega^{n-1})$.
Note that, in the definition with $B_f$, if we interchange two
coordinates, then $dx_1\wedge\ldots \wedge dx_n$ changes sign, but also
the order the $f_i$ is changed in a corresponding way. As a result, the
signature of $B_f$ does not change. 
\begin{example}
It is easy to verify that $\sigma$ has the following properties:\\
1) $\sigma(-f)=(-1)^n \sigma(f)$.\\
2) If $f$ has at most one-dimensional critical locus, and $g$ has an 
isolated critical point, then
\[ \sigma(f \oplus g)=\sigma(f) \cdot \sigma(g)\]
Here $\oplus$ denotes the Thom-Sebastiani sum of $f$ and $g$.
Furthermore, one computes $\sigma(x^2)=1$, hence $\sigma(-x^2)=-1$, and
so $\sigma(x^2+y^2)=\sigma(-(x^2+y^2)=1$, whereas $\sigma(x^2-y^2)=-1$.
For $x^2y \in \R[[x,y]]$ one finds 
\[\sigma(x^2y)=1\]
so that $\sigma(x^2y+z^2)=1$, $\sigma(x^2y-z^2)=-1$. Note that
$x^2y-z^2=-(x^2(-y)+z^2)$, but indeed $\sigma(-f)=-\sigma(f)$ for
a function of three variables.
\end{example}
 
\begin{theorem} \label{signatureconstant}
Let $f \in \R\{x_1,x_2,\ldots,x_n\}$ be a function with a 
one-dimensional critical locus $\Sigma$ and let $F:X \times S \lra \R \times S$ be
a good representative of an admissible deformation of $f,\sigma$. Then for
a $s \in S$ one has
\[\sigma(f)=\sum_{p \in X} \sigma(F_s,p)\]
\end{theorem}
\begin{proof}
We denote by $\pi: X \times S \lra S$ the canonical projection and
consider as before the sheaf ${\cal M}_S={\cal I}_S/{\cal J}_S$.
Recall that $\pi_*({\cal M}_S)$ is a free ${\cal O}_S$-module of
finite rank, equal to $\dim(I/J)$. Using \ref{pairingfamily} we obtain a family of 
pairings
\[ B_S:\pi_*({\cal M}_S) \times \pi_*({\cal M}_S) \lra {\cal O}_S\]
parametrised by $s \in S$. Let $B_{S,s}$ the resulting pairing on the
fibre $\pi_*({\cal M}_S)_s$. The
function
\[Signature: S \lra \Z,\;\;s \mapsto Signature(B_{S,s})\]
is constant, as in a basis it is described as the signature of a non-singular
matrix that depends holomorphically on $s \in S$. For fixed $s \neq 0$, this
matrix appears in block diagonal form, corresponding to the singularities of
the function $F_s$. Note also that complex singularities appear in 
complex conjugate pairs, whose contribution to $\sigma$ turn out to cancel
 each other. The result follows.
\end{proof}

\begin{theorem}\label{top}
Let $f \in \R\{x_1,x_2,\ldots,x_n\}$ have a one-dimensional critical 
locus which admits a morsfication. Let $F_{\pm}$ denote the positive and
negative Milnor fibre of $f$. Then one has:
\[
2 \sigma(f)=-\widetilde{\chi}(F_{-})+(-1)^{n-1}\widetilde{\chi}(F_{+})
\]
\end{theorem}

\begin{proof}
We apply the previous result to a (good representative of a) morsification
 $F: X \times S \lra \R \times S$. For $s \neq 0$ the singularities of the 
function $F_s$ are of typ $A_1$, $A_{\infty}$ and $D_{\infty}$. For
each case there are different real forms to consider. Furthermore, complex
singularities appear in complex conjugate pairs, whose contribution to
$\sigma$ turn out to cancel each other. The result then follows 
from the truth for these singularities, an easy cut and paste argument.
More details can be found in \cite{egodiss}.
\end{proof}

\begin{figure}
\begin{tabular}{ccc}
\epsfig{figure=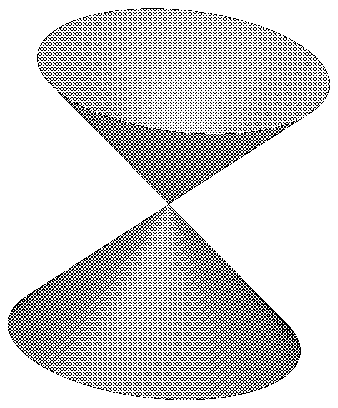}&\epsfig{figure=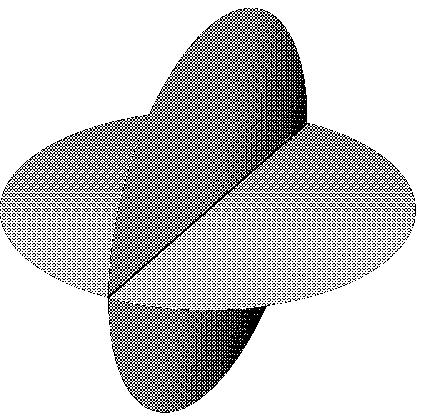}&\epsfig{figure=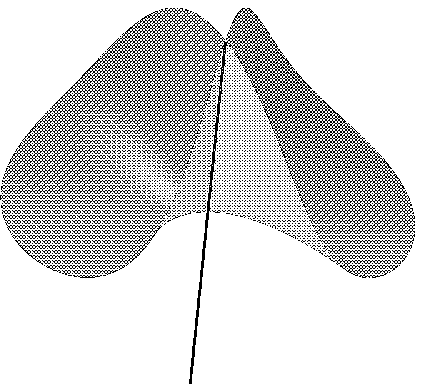}\\
$A_1$&$A_{\infty}$&$D_{\infty}$
\end{tabular}
\caption{Morse type of non-isolated singularites}
\end{figure}

\begin{remark}
We conjecture the above formula to hold for all singularities with critical
locus a curve $\Sigma$ and with transverse type $A_1$. However, the above
proof does not apply, in particular not if the curve $\Sigma$ is not smoothable.
\end{remark}

\begin{theorem}\label{threespace}
The formula of theorem \ref{top} is valid for all non-isolated hypersurface singularities in three space with one-dimensional critical locus of transverse type $A_1$:
\[2 \sigma(f)=\chi(F_{+})-\chi(F_{-})\]
\end{theorem}
\begin{proof}
Let $\Sigma \subset (\C^3,0)$ be the critical locus of $f$, described by a
radical ideal $I$, so $f \in \int I$ and let $g \in I^2$ be generic. 
The statement holds for $g$, because for these a morsification does exist 
and thus we can apply the previous result.
Also, there is an admissible deformation described by $F=tf+(1-t)g$ and
where the curve $\Sigma$ is deformed trivially. By \ref{signatureconstant}
we have $\sigma(g)=\sigma(f)+\sum \sigma_{i}(f_i,p_i)$, where the $(f_i,p_i$
are the germs of singularties (of type $A_1$ of $D_{\infty}$ that split off 
in this deformation. As we know the truth for these singularities, we can 
conclude the truth of the statement for $f$ itself.
\end{proof}

\begin{remark}
(i) The formula \ref{threespace} appears to holds for a much broader 
class of surface singularities with other transverse types, but it is not 
true  without further assumptions. For example, the jacobian module $I/J_f$ 
for the sextic $f=(x^2+y^2)^3-4x^2y^2z^2=0$ has $2t^5+3t^6+2t^7$ as 
Poincar\'e polynomial. So its dimension is seven, and thus
the signature in any case is odd. In fact it can be computed to be equal to 
three. In contrast, the difference of the real Euler characterisics is eight.\\
(ii) From additivity of the Euler number one has
\[ \chi(F_{+})+\chi(F_{-})+\chi(F_0)=2\]
where $\chi(L_0)$ is the Euler number of the real link 
$L_0:=S^{n-1} \cap F_0$ of the singularity. In turn, this number is equal
to 
\[\chi(L_0)=\sum_i \sigma(\Sigma_i)\]
where the sum runs over the {\em real half-branches} $\Sigma_i$ of the curve
$\Sigma$, and where $\sigma(\Sigma_i)$ denotes the $\sigma$ of the singularity
transverse to $\Sigma_i$. Formulated differently, 
\[\sum_i \sigma(\Sigma_i)=N-D\]
where $N$ denotes the nunmber of {\em naked half-branches} (transverse type
$\pm(x^2+y^2)$) and $D$ the number of {\em dressed half-branches} (transverse
type $x^2-y^2$). 
\end{remark}

\subsection{A signature theorem for projective curves}
A homogenous polynomial $f \in \R[x,y,z]$ defines a cone $X:=\{f=0\}\subset \R^3$
which is the same as a projective curve $ C \subset \P^2(\R)$. Furthermore,
if $f$ has even degree, $f$ has a well defined {\em sign} for each point in the 
complement of the curve. We put $V_+:=\{(x:y:z)\in\P^2(\R):f(x,y,z)>0\}$ and
 $V_-:=\{(x:y:z)\in \P^2(\R):f(x,y,z)<0\}$
The singularities of the cone $X$ is a finite union of lines, namely the cone over
the singularities of $C$. The curve $C$ has ordinary double points, precisely if
the the transverse type outside the origin ist $A_1$.
\begin{theorem}
Let $f\in\R[x,y,z]$ be a homogenous polynomial of even degree defining a curve
with only double points as singularities. Then
\[
\sigma_f=\chi(V_+)-\chi(V_-).
\]
\end{theorem}
\begin{proof} (see also \cite{arnold}) Let $F_+$ and $F_-$ denote the positive and negativ Milnor fibre of $f$ in $\R^3$. The restriction of the canonical projection $\R^3\setminus\{0\}\rightarrow \P^2(\R)$ to the Milnor fibres $F_+\rightarrow V_+$ and $F_-\rightarrow V_-$ is an unramified $2:1$ covering. For the Euler characteristics we thus obtain
$\chi(F_+)=2\chi(V_+)$, $\chi(F_-)=2\chi(\V_-)$.
hence  the result follows from \ref{threespace}.
\end{proof}

\begin{example}
The signature of the cone over the nodal quartic defined by
$$(x^2+y^2)^2+3x^2y-y^3=0$$
 is $\sigma=-3$, we see $\chi_+=0$ and $\chi_-=3$.
The signature of the cone over the quartik with a $D_4$-singularity defined by
$$2(x^2+y^2)^2+10y(3x^2-y^2)+11(x^2+y^2)-3=0$$

\begin{figure}
\begin{center}
\epsfig{figure=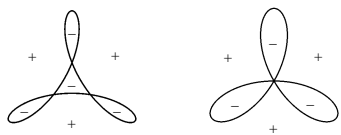}
\end{center}
\caption{\label{quad}Two Quartics}
\end{figure}
is $\sigma=-4$, the Euler numbers are $\chi_+=0$ and $\chi_-=4$.
Figure \ref{quad} shows the curves and the defined regions.
\end{example}

\begin{example}
Depending on the situation, the signature can be used to find real components of algebraic curves.
\begin{figure}
$$
\begin{array}{ccc}
\epsfig{figure=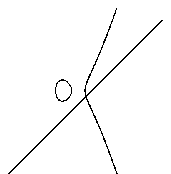}&\epsfig{figure=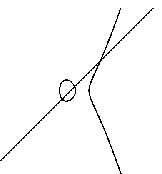}&\epsfig{figure=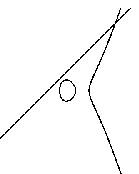}\\
\sigma=2&\sigma=0&\sigma=-2
\end{array}
$$
\caption{\label{test}
Relative position of a line and a cubic with signature of product quartic}
\end{figure}
In figure \ref{test}, the signature of the quartic, consisting of the shown cubic and a moving test line, separates the different topological positions of the line relative to the cubic.
\end{example}

\begin{example}
In figure \ref{zweiqu} 
the signature of the shown quartic separates the relative topological configuration of the two quadrics.

\begin{figure}
$$
\begin{array}{ccccc}
\epsfig{figure=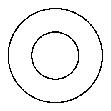}&\epsfig{figure=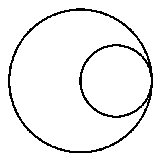,width=1.7cm}&\epsfig{figure=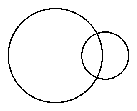}&\epsfig{figure=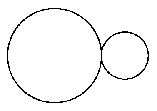}&\epsfig{figure=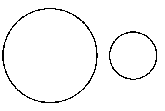}\\
\sigma=1&\sigma=0&\sigma=-1&\sigma=-2&\sigma=-3
\end{array}
$$
\caption{\label{zweiqu}
Relative position of two quadrics and signature of the product quartic}
\end{figure}
\end{example}

\section{Some open problems}

We have shown that the self-duality of the Koszul-complex leads in the almost complete intersection case to a self-duality of $I/J$. In the real case one can define a signature. For the 
special case of Jacobi-modules of hypersurface singularities with one-dimensional singular locus 
this signature was related, in special cases, to Euler-characteristics of positive and negative Milnor fibre, which can be seen as a generalisation of the theorem formulated in \cite{arnold}.\\

$\bullet$ What is the topological meaning of the signature in general? In other words, proper generalisation of the Eisenbud-Levine theorem? 
We have seen that for Jacobi-modules of a special class of 
hypersurfaces with one-dimensional singular locus there is a direct
relation with the Euler characteristic of the positive and negative Milnor fibre. But we have also seen that this relation does not hold in all examples. 
The more special question is: for exactly what class of singularities with 
one-dimensional singular locus is theorem \ref{top} is true?\\

$\bullet$ In the complete intersection case the pairing on $P/J$ factors over the multiplication map. For an almost complete intersection the non-degenerate 
pairing 
\[I/J \times I/J \lra K\]
is also $P$-linear, but there is no evident 'multiplication map' to factor over.

\begin{conjecture} In the case where $I$ is a radical ideal one has:

1) The pairing on $I/J$ factors over the multiplication map
\[I/J \times I/J \lra I^2/IJ\]

2) If $I/J \neq 0$ then The $P$-module $I^2/IJ$ has a one-dimensional socle.\\

3) The socle of $I^2/IJ$ is generated by the Jacobian determinant
\[ h:=\left|\begin{array}{ccc} \frac{\partial f_1}{\partial x_1}& \ldots &\frac{\partial f_1}{\partial x_n}\\
\ldots&\ldots&\ldots\\
\frac{\partial f_n}{\partial x_1} &\ldots& \frac{\partial f_n}{\partial x_n}\\
 \end{array}\right|\] 

\end{conjecture}
We note that for radical ideals $I$ in three variables one always has
$\dim(I/J)=\dim(I^2/IJ)$, by \cite{dJvS1}
For non-radical ideals $I$ the statements of the conjecture often hold, 
but again the function $f=(x^2+y^2)^3-x^2y^2z^2$ provides a counterexample.
Here $I^2/IJ$ has Poincar\'e polynomial $2t^{10}+2t^{11}$, however, the
Hessian $h$ sits in degree $12$. The series $(x^3+y^3)^p-x^py^pz^p$ provides
further examples.
 
The pairing on $I/J$ is {\em symmetric}, but we do not know a good algebraic proof of this fact. Of course it would be explained by the above conjecture.\\

$\bullet$ Is there a way to exhibit a self-dual resolution resolution of $I/J$ as $P$-module?\\
If 
\[0\lra F_{n-1} \lra \ldots \lra F_1 \lra P \lra P/I\]
is a free resolution over $P$, the map $P/J \lra P/I$ can be covered by a map of
the Koszul-complex $\wedge^{\bullet}$ to $F_{\bullet}$ giving a diagram
\[
\begin{array}{cccccccccc}
\wedge^{n-1} &\lra&  \wedge^{n-2}& \lra &\ldots &\wedge^1&\lra &P& \lra &P/J\\
\downarrow&&\downarrow&&&\downarrow&&\downarrow&&\downarrow\\
F_{n-1}& \lra&F_{n-2}& \ldots& \lra& F_1& \lra& P& \lra & P/I\\
\end{array}
\]
Applying $Hom_P(-,P)$ to the diagram, and using $Hom(\wedge^p,P)=\wedge^{n-p}$
we obtain also maps
\[
\begin{array}{ccccccccc}
F_{1}^*& \lra&F_{2}^*&\lra& \ldots& \lra&F_{n-2}^*&\lra & F_{n-1}^*\\
\downarrow&&\downarrow&&&&\downarrow&&\downarrow\\
\wedge^{n-1} &\lra&  \wedge^{n-2}& \lra &\ldots&\lra& \wedge^2&\lra&\wedge^1\\
\end{array}
\]

These two diagrams can be put on top of each other: 
\[
\begin{array}{ccccccccc}
F_{1}^*& \lra&F_{2}^*&\lra& \ldots& \lra&F_{n-2}^*&\lra & F_{n-1}^*\\
\downarrow&&\downarrow&&&&\downarrow&&\downarrow\\
\wedge^{n-1} &\lra&  \wedge^{n-2}& \lra &\ldots&\lra& \wedge^2&\lra&\wedge^1\\
\downarrow&&\downarrow&&&&\downarrow&&\downarrow\\
F_{n-1}& \lra&F_{n-2}&\lra& \ldots& \lra&F_{2}&\lra & F_{1}\\
\end{array}
\]
If this were a double complex, the total complex would be a self-dual resolution
of $I/J$ of the right length, see \cite{pellikaan3}.
Unfortunately, is is not clear that the maps always can be chosen as to obtain a 
double complex.

$\bullet$ Furthermore, one might ask for generalisations going further than almost complete intersections
and look for self-dual pieces in the Koszul-homology. It is not clear how such a 
theorem might look like: there are simple examples of sequences
$f_1,\ldots,f_n$ in $P$ with $\dim(P/J)=2$, but for which $H^0_{\m}(P/J)$ is not a Gorenstein module. The homogeneous function
$f=(x+y)(xu)^2+(u+v)(xv)^2+(x+u+v)(yu)^2+(y+u)(yv)^2$ is singular along the union of two planes defined by the ideal $(x,y) \cap (u,v)$, so the locus defined by $J=J_f$ is two-dimensional. The module $H^0_{\m}(P/J)$ is computed to have
\[6t^4+13t^5+15t^6+9t^7\]
as Poincar\'e series, so it is not Gorenstein.\\

{\bf Acknowledgement:} The first author wants to thank D. Eisenbud for showing interest in the pairing in an early stage of this work and asked the question as to the meaning of the signature in the real case. The work was part of the Ph. D. thesis of the second author. The authors thank further R. Pellikaan exchange of ideas and T. de Jong for suggesting the argument used in \ref{threespace}.


\bibliographystyle{apalike}
\bibliography{math}
\end{document}